\documentclass[11pt]{amsart}

\usepackage{amsmath}
\usepackage{amsthm}
\usepackage{amssymb}
\usepackage{mathrsfs}
\usepackage{verbatim}

\usepackage{xcolor}
{%
	\setcounter{enumi}{0}

	\begin{enumerate}}%
	{\end{enumerate} }

{%
	\setcounter{enumi}{0}

	\begin{enumerate}}%
	{\end{enumerate} }

\usepackage[colorlinks]{hyperref}

\newtheorem{theorem}{Theorem}[section]

\newtheorem{proposition}[theorem]{Proposition}

\theoremstyle{definition}

\numberwithin{equation}{section}

\allowdisplaybreaks

\newcommand*\re{\mathbb{R}}
\newcommand*\rn{\mathbb{R}^n}

\newcommand*\omegal{\Omega_\ell}

\begin{document}
	
	\title[General p-Laplacian eigenvalue problem]{ON THE ASYMPTOTIC BEHAVIOR OF THE EIGENVALUES OF NONLINEAR ELLIPTIC PROBLEMS IN DOMAINS BECOMING UNBOUNDED}

	\maketitle
	\centerline{\scshape Luca Esposito$^1$, Prosenjit Roy$^2$ and Firoj Sk$^2$}
	
	\medskip 
	{\footnotesize
		
		\centerline{1 DipMat, University of Salerno, Italy.}
		\centerline{E-mail: luesposi@unisa.it}
		\centerline{2 Indian Institute of Technology,  Kanpur, India.}
		\centerline{E-mail: prosenjit@iitk.ac.in and firoj@iitk.ac.in}
		
	}

	\begin{abstract}
		We analyze the asymptotic behavior of the eigenvalues of nonlinear elliptic problems under Dirichlet boundary conditions and mixed (Dirichlet, Neumann) boundary conditions on domains becoming unbounded. We make intensive use of Picone identity to overcome nonlinearity complications. Altogether the use of Picone identity makes the proof easier  with respect to the known proof in the linear case. Surprisingly the asymptotic behavior under mixed boundary conditions critically differs from the case of pure Dirichlet boundary conditions for some class of problems.
	\end{abstract}
     
     \smallskip
     
	\vspace{.5cm}
	\section{Introduction}\label{section:introduction}
	In this paper we study nonlinear elliptic eigenvalue problems on domains which become unbounded  in one or several directions. We have basically focused on operators related to the p-Laplacian. To be more precise let us introduce some notations that we will use in the rest of the paper. \smallskip
	
	Let $1\leq m< n$ and let 
	$\omega_1$, $\omega_2$ be two open bounded sets in $\mathbb R^m$ and $\mathbb R^{n-m}$ respectively. For every $\ell>0$ let us define $\Omega_\ell$:=$\ell\omega_1\times \omega_2$. We will denote, for every $x\in \re^n$
	$$
	x=(X_1,X_2),
	$$
	with
	$$
	X_1=(x_1,\dots, x_m), \hspace{5mm} X_2=(x_{m+1},\dots, x_n\\
	).
	$$
	$\nabla, \nabla_{X_1}$ and $\nabla_{X_2}$ will denote gradient vectors in $\mathbb{R}^n, \mathbb{R}^m$ and $\mathbb{R}^{n-m}$ respectively. Let $A=A(X_1,X_2)$ be an $n\times n$-symmetric matrix of the type 
	\[
	A=
	\begin{pmatrix}
	A_{11}(x)     & A_{12}(X_2)\\
	A_{12}^t(X_2) & A_{22}(X_2)
	\end{pmatrix}
	,\]
	where $A_{22}$ is an $(n-m)\times(n-m)$ matrix.
	We will assume that $A$ is an uniformly bounded and uniformly positive definite matrix on $\mathbb R^m\times\omega_2$. Precise conditions on the matrix $A$ will be clarified in section \ref{dirichlet}. We start considering the following eigenvalue problem with Dirichlet boundary condition for any $p\geq 2$,

	\begin{equation}\label{dirichlet case}
	\begin{cases}
	-\text{div}\big(|A(x)\nabla u_\ell\cdot\nabla u_\ell|^\frac{p-2}{2}A(x)\nabla u_\ell\big)=\lambda_D(\omegal)|u_\ell|^{p-2}u_\ell\;\;\;\text{  in }\Omega_\ell,\\ u_\ell=0\;\;\;\text{ on }\partial\Omega_\ell.
	\end{cases}
	\end{equation}
	We are interested in the study of the asymptotic behaviour of the first eigenvalue $\lambda_D^1(\Omega_\ell)$ of the above problem as $\ell\rightarrow \infty$.  
	In the linear case $p=2$, in a seminal paper of Chipot and Rougirel (see \cite{arn}), it was proved that the k-th eigenvalue of \eqref{dirichlet case}(see \cite{arn} for the definition of k-th eigenvalue) converges to the first eigenvalue of the corresponding cross section problem that we now introduce in the general case $p\geq 2$,

	\begin{equation}\label{dirichlet cross section}
	\begin{cases}
	-\text{div}\big(|A_{22}(X_2)\nabla u\cdot\nabla u|^\frac{p-2}{2}A_{22}(X_2)\nabla u\big)=\mu(\omega_2)|u|^{p-2}u\;\;\;\text{  in }\omega_2, \\ u=0\;\;\;\text{ on }\partial\omega_2.
	\end{cases}
	\end{equation}
	We will denote by $\mu_1(\omega_2)$ and $W$ respectively the first eigenvalue and the first normalized ($ ||W||_{L^p(\omega_2)} =1$)  eigenfunction of \eqref{dirichlet cross section}.    In the first part of the paper we are able to prove that also in the nonlinear case $p>2$ the first eigenvalue of \eqref{dirichlet case} converge to the  first eigenvalue of the problem (\ref{dirichlet cross section}) on cross section.  Moreover, we would like to introduce the following minimization problem on the infinite strip $\Omega_\infty:= \mathbb{R}^m \times \omega_2  $ in order to get a deeper insight of the asymptotic behaviour of $\lambda^1_D(\Omega_\ell)$:
	\begin{equation*}
	\label{limitinf}
	\Lambda_\infty = \inf_{\substack{u \in W_0^{1,p}(\Omega_\infty),\\ u\neq 0}} \frac{\int_{\Omega_\infty} |A\nabla u\cdot \nabla u|^{\frac{p}{2}} }{\int_{\Omega_\infty}|u|^p}.
	\end{equation*}
	
	Precisely we prove the following theorem.
	
	\begin{theorem}\label{conv for dirichlet case}
		Let $p\geq 2$ and $\mu_1(\omega_2)$ denote  the first eigenvalue of (\ref{dirichlet cross section}), then there exists a constant $C$ depending only on $A,\omega_2, p$, such that
		$$
		\mu_1(\omega_2)\leq\lambda^1_D(\omegal)\leq\mu_1(\omega_2)+\frac{C}{\ell},
		$$
		for every $\ell >0$. Moreover $\Lambda_\infty = \mu_1(\omega_2).$
	\end{theorem}
	
	The lower bound in Theorem \ref{dirichlet case} was proved by Chipot and Rougirel (see \cite{arn}), in the linear case $p=2$, using an approximation argument for the matrix $A$ that relies on the linearity of the  equation and it is not clear if the same argument can be employed in the non linear case. In this paper we present a complete different argument  that relies  on a clever use of  Picone identity (see Theorem \ref{picone}). 
	In spite of the difficulty of non linearity our  approach turns out  to be  simpler and shorter than \cite{arn}. The upper bound in Theorem \ref{dirichlet case} is obtained, in a similar manner as in (\cite{arn}), by constructing a suitable test function on the truncated domains $\Omega_{\frac{\ell}{2}}=\frac{\ell}{2}\omega_1\times\omega_2$ and then letting $\ell$ tending to infinity.
	\smallskip
	
	The second part of the paper concerns the  eigenvalue problem for  the same operator  as in  \eqref{dirichlet case}, but with  mixed boundary conditions. For technical reasons (which precisely the construction of test function ``$\phi_\ell$" in the proof of Theorem \ref{mixed case ell to infinity} ), we can only allow the domain $\Omega_\ell$ to become unbounded in  one direction, i.e. we assume that $\omega_1 = (-1,1)$ and $A_{11}(x) = a_{11}(X_2)$. Namely we consider the following eigenvalue problem on  $\omegal=(-\ell,\ell)\times\omega_2$:
	\begin{equation}\label{mixed eigenvalue}
	\begin{cases}
	-\text{div}\big(|A(X_2)\nabla u_\ell\cdot\nabla u_\ell|^\frac{p-2}{2}A(X_2)\nabla u_\ell\big)=\lambda_M(\omegal)|u_\ell|^{p-2}u_\ell\;\;\;\text{  in }\omegal,\\ u_\ell=0\;\;\;\text{ on }\gamma_\ell:= (-\ell,\ell )\times \partial\omega_2,\\
	(A(X_2)\nabla u_\ell)\cdot\nu=0\;\;\;\text{ on }\Gamma_\ell:= \{ -\ell,\ell\}\times \omega_2,
	\end{cases}
	\end{equation}
	where $\nu$ denotes the outward unit normal to $\Gamma_\ell$.  For the case $p=2$, in Chipot, Roy and Shafrir \cite{crs} it was proved that when $\ell$ goes to plus infinity the limit of the first eigenvalue $\lambda_M^1(\omegal)$ exists.   In addition this limit is  strictly smaller than $\mu_1(\omega_2)$ if and only if   $A_{12}\cdot\nabla_{X_2} W\neq 0$ a.e. on $\omega_2$.  In the nonlinear case $( p \geq 2)$, we prove that this gap phenomenon still holds under the same condition  $A_{12}\cdot\nabla_{X_2} W\neq 0$ a.e. on $\omega_2$. In particular we prove  the following theorem. 
	\begin{theorem}\label{mixed case ell to infinity}
		For $p\geq 2,$
		we have 
		$$
		\lim_{\ell\to\infty}\sup\lambda^1_M(\omegal)<\mu_1(\omega_2),
		$$
		provided $A_{12}\cdot\nabla_{X_2} W\neq 0$ a.e. on $\omega_2$, otherwise $\lambda^1_M(\omegal)=\mu_1(\omega_2)$ for all $\ell>0.$
	\end{theorem}
	
	The  main steps to  prove the above theorem uses the same argument as in \cite{crs}. 	The first step is to study a ``dimension reduction" problem, namely we  let $\ell$ go to zero in \eqref{mixed eigenvalue}. As a matter of fact it turns out that  $\lim_{\ell \rightarrow 0} \lambda_M^1(\omegal) < \mu_1(\omega_2)$  if and only if $A_{12}\cdot\nabla_{X_2} W\neq 0$ a.e. on $\omega_2$. This provides us the main tool to  construct  test functions on $\Omega_\ell$   in order to prove the gap phenomenon for large values of $\ell$. We address the readers to  \cite{abp}, \cite{br}, \cite{ff}  and the references there in, for the general study of problems on ``dimension reduction".

	Asymptotic behavior when the parameter $\ell\to\infty$ for different type of problems subject to different boundary conditions were studied in past. We refer to \cite{CS1} and \cite{CS2}  for the study of Stokes problem and elliptic equations with Neuman boundary conditions.  Asymptotic behaviour for the minimizers of purely variational problem is done in \cite{mojsic}, \cite{ldm}.  We refer to \cite{am}, \cite{new1}, \cite{chipot1}, \cite{chipot2}, \cite{b}, \cite{arn1}, \cite{ka}, \cite{karen}, \cite{pi}, \cite{pi1}, \cite{pi2}, \cite{new2}, \cite{Sen1}, \cite{Sen2}, \cite{karen1} and the reference mentioned there in  for other related work in this direction. 
	
	\smallskip
	
	\section{Some Preliminary results}
	In this section we will summarize some standard features about eigenvalues and eigenfunctions of p-Laplacian type operators. Let $\Omega$ be a bounded open regular subset of $\mathbb R^n$, $1<p<\infty$, we will denote by $W^{1,p}(\Omega)$, $W^{1,p}_0(\Omega)$ the usual spaces of functions defined by 
	
	$$
	W^{1,p}(\Omega)=\{v\in L^p(\Omega):\partial_{x_i}v\in L^p(\Omega), i=1,2,\cdots,n\},
	$$
	equipped with the norm 
		\begin{equation*}
	|| v||_{p,\Omega}=\bigg\{\int_{\Omega}\bigl(|v|^p+|\nabla v|^p\bigr)\bigg\}^{1/p},
	\end{equation*}
	and
	
	$$
	W^{1,p}_0(\Omega)=\{v\in W^{1,p}(\Omega):v=0\text{ on }\partial\Omega\}.
	$$
	
	Thanks to the classical Poincar\'e inequality, we will always assume that the space $W^{1,p}_0(\Omega)$ is
	equipped with the norm 
	
	\begin{equation}\label{norm}
	|| v||_{p,\Omega}^p=\int_{\Omega}|\nabla v|^p.
	\end{equation}

	Let $A(x)$ be a symmetric, uniformly positive definite and uniformly bounded matrix in $\rn$. We assume that $A(x)$ is $C^1$ i.e. each component is $C^1.$ Therefore we consider the following Dirichlet eigenvalue problem:

	\begin{equation}\label{gen p}
	\begin{cases}
	-\text{div}\big(|A\nabla u\cdot\nabla u|^\frac{p-2}{2}A\nabla u\big)=\lambda(\Omega)|u|^{p-2}u\;\;\;\text{  in }\Omega,\\ u=0\;\;\;\text{ on }\partial\Omega.
	\end{cases}
	\end{equation}
	 For reader convenience we will refer to the weak formulation of \eqref{gen p}, which consists in the following:
	\begin{equation*}\label{dirichlet weak}
	\begin{cases}
	u\in W^{1,p}_0(\Omega),\\
	\displaystyle\int_{\Omega}|A\nabla u\cdot\nabla u|^\frac{p-2}{2}A\nabla u\cdot\nabla v=\lambda(\Omega)\displaystyle\int_{\Omega}|u|^{p-2}uv\;\;\text{ for any } v\in W^{1,p}_0(\Omega).
	\end{cases}
	\end{equation*}
	We denote by $\lambda^1_D(\Omega)$ and $u_0$ the first eigenvalue and the first eigenfunction of \eqref{gen p} respectively. Now we collect some properties of the first eigenpair $(\lambda^1_D(\Omega),u_0)$.
\begin{proposition}
The following properties verified by $\lambda^1_D(\Omega)$ and $u_0$ 
	\begin{itemize}
		\item $\lambda_D^1(\Omega)$ is finite and strictly positive.
		\item $\lambda_D^1(\Omega)$ fulfill the following variational chracterization by means of the Rayleigh quotient,
		          \begin{align}\label{eigen value}
		          \lambda^1_D(\Omega)
		          =\inf_{\substack{u\in W^{1,p}_0(\Omega)\\u\neq 0 }}\frac{\int_{\Omega}|A\nabla u\cdot\nabla u|^{\frac{p}{2}}}{\int_{\Omega}|u|^p}=\frac{\int_{\Omega}|A\nabla u_0\cdot\nabla u_0|^{\frac{p}{2}}}{\int_{\Omega}|u_0|^p}.
		          \end{align}
		 \item $u_0$ is bounded and is in $C^{1,\gamma}(\Omega)$ for some $\gamma>0$.
		 \item $\lambda_D^1(\Omega)$ is simple and the function $u_0$ does not change sign in $\Omega.$
	\end{itemize}
\end{proposition}
	In the case of pure p-Laplacian eigenvalue problem the properties listed above are well known. The reader is addressed to \cite{huang}, \cite{anle}, \cite{lindqvist}, \cite{tolk}, \cite{trudinger} for the proof. In the general case of equation \eqref{gen p} the same results can be proved with obvious slight modifications.
	
	One of the main tools for the proof of the main result of this paper is the following Picone's identity. For the sake of completeness we give here the proof of this fundamental inequality (see also \cite{huang}).
	\begin{theorem}\textbf{(Picone's identity)}\label{picone}
		Suppose $A$ is a symmetric positive definite matrix on $\mathbb R^n$ and $u,v$ two differentiable functions with $u\geq 0$ and $v>0$. Define 
		$$
		L(u,v)=|A\nabla u\cdot\nabla u|^\frac{p}{2}-\frac{pu^{p-1}|A\nabla v\cdot\nabla v|^\frac{p-2}{2}(A\nabla v\cdot\nabla u)}{v^{p-1}}+\frac{(p-1)u^p|A\nabla v\cdot\nabla v|^\frac{p-2}{2}(A\nabla v\cdot\nabla v)}{v^p},
		$$
		and
		$$R(u,v)=|A\nabla u\cdot\nabla u|^\frac{p}{2}-\nabla\bigg(\frac{u^p}{v^{p-1}}\bigg)|A\nabla v\cdot\nabla v|^\frac{p-2}{2}A\nabla v.
		$$
		Then $L(u,v)=R(u,v)\geq 0$. Moreover $L(u,v)=0$ a.e. in $\Omega$ if and only if $\nabla(\frac{u}{v})=0$ a.e. in $\Omega$.
	\end{theorem}
	\begin{proof}
	The equality case $L(u,v)=R(u,v)$ trivially follows by expanding $R(u,v)$. Now by hypothesis on the matrix $A$, we can write $|A\nabla u\cdot\nabla u|=||A^\frac{1}{2}\nabla u||^2$. Thus
		\begin{align*}
		L(u,v)
		&=|A\nabla u\cdot\nabla u|^\frac{p}{2}-\frac{pu^{p-1}|A\nabla v\cdot\nabla v|^\frac{p-2}{2}(A\nabla v\cdot\nabla u)}{v^{p-1}}+\frac{(p-1)u^p|A\nabla v\cdot\nabla v|^\frac{p-2}{2}(A\nabla v\cdot\nabla v)}{v^p}\\
		&=||A^\frac{1}{2}\nabla u||^p+(p-1)\frac{u^p}{v^p}||A^\frac{1}{2}\nabla v||^p-\frac{pu^{p-1}||A^\frac{1}{2}\nabla v||^{p-2}(A\nabla v\cdot\nabla u)}{v^{p-1}}\\
		&=p\bigg(\frac{||A^\frac{1}{2}\nabla u||^p}{p}+\frac{u^p||A^\frac{1}{2}\nabla v||^p}{qv^p}\bigg)-\frac{pu^{p-1}||A^\frac{1}{2}\nabla v||^{p-2}(A\nabla v\cdot\nabla u)}{v^{p-1}}\\
		&=p\bigg(\frac{||A^\frac{1}{2}\nabla u||^p}{p}+\frac{(u||A^\frac{1}{2}\nabla v||)^{(p-1)q}}{qv^{(p-1)q}}\bigg)-\frac{pu^{p-1}||A^\frac{1}{2}\nabla u||\;||A^\frac{1}{2}\nabla v||^{p-1}}{v^{p-1}}\\
		&\;\;\;\;\;+\frac{pu^{p-1}||A^\frac{1}{2}\nabla v||^{p-2}}{v^{p-1}}\big(||A^\frac{1}{2}\nabla u||\;||A^\frac{1}{2}\nabla v||-A\nabla u\cdot\nabla v\big).
		\end{align*}
	   Using Young's inequality we have,
	   
	   $$\frac{||A^\frac{1}{2}\nabla u||^p}{p}+\frac{(u||A^\frac{1}{2}\nabla v||)^{(p-1)q}}{qv^{(p-1)q}}\geq\frac{u^{p-1}||A^\frac{1}{2}\nabla u||\;||A^\frac{1}{2}\nabla v||^{p-1}}{v^{p-1}}
	   $$
	   where $1/p+1/q=1.$ Equality holds when $$||A^\frac{1}{2}\nabla u||=\frac{u}{v}||A^\frac{1}{2}\nabla v||.$$Therefore we get $L(u,v)\geq 0.$ So when, $L(u,v)(x_0)=0$ and $u(x_0)\neq 0$ we must have $||A^\frac{1}{2}\nabla u||=\frac{u}{v}||A^\frac{1}{2}\nabla v||$ and $||A^\frac{1}{2}\nabla u||\;||A^\frac{1}{2}\nabla v||=A\nabla u\cdot\nabla v$ and thus we obtain $A^\frac{1}{2}\nabla u=\frac{u}{v}A^\frac{1}{2}\nabla v$. Therefore, $\nabla\big(\frac{u}{v}\big)(x_0)=0.$ On the other hand if, $u(x_0)=0$ then $\nabla u=0$ a.e. on $\{u(x)=0\}$ and thus $\nabla\big(\frac{u}{v}\big)=0$ a.e. on $\{u(x)=0\}$. Therefore we conclude that $\nabla(\frac{u}{v})=0$ a.e. in $\Omega.$
	\end{proof}

\section{Convergence of the first eigenvalue for Dirichlet case.}\label{dirichlet}
	Resuming the notation used in the introduction we denote $\Omega_\ell=\ell\omega_1\times\omega_2$ to be an open subset of $\mathbb R^n$, where $\omega_1$, $\omega_2$ are two open bounded sets in $\mathbb R^m$, and $\mathbb R^{n-m}$ respectively and $\ell >0$. The variables in $\omega_1$, and $\omega_2$ are denoted by $X_1$ and $X_2$ respectively. We will write $x=(X_1,X_2)\in \mathbb R^{m}\times \mathbb R^{n-m} $ accordingly.
	
	The matrix 
	\[
	A=
	\begin{pmatrix}
	A_{11}(x)     & A_{12}(X_2)\\
	A_{12}^t(X_2) & A_{22}(X_2)
	\end{pmatrix}
	,\]
	is an $n\times n$-symmetric matrix, and assume that the block matrix $A_{22}$ is $C^1$ regularity.
	
We will assume that $A$ is uniformly bounded and uniformly positive definite matrix on $\mathbb R^m\times\omega_2$; namely there exists two positive constants $M$, $\lambda$ such that $$||A(x)||\leq M\;\;\text{ a.e. }x\in\mathbb R^m\times\omega_2,$$ $$A(x)\xi\cdot\xi\geq\lambda|\xi|^2\;\;\; \text{ a.e. }x\in\mathbb R^m\times\omega_2,\;\;\forall\xi\in\mathbb R^n.$$
In the following $||\cdot||$ will denote the norm of matrices, $|\cdot|$ the euclidean norm, and  $``\cdot"$ the usual euclidean scalar product. 
	
In this section we investigate the asymptotic behaviour of the first eigenvalue $\lambda^1_D(\Omega_\ell)$ of the problem \eqref{dirichlet case} for $\ell \rightarrow \infty$. Indeed we prove Theorem 1.1 which claim that $\lambda^1_D(\Omega_\ell)$ converges to the first eigenvalue $\mu_1(\omega_2)$ of the problem \eqref{dirichlet cross section}. For the reader convenience we quote the weak formulation of the problem  \eqref{dirichlet cross section}.

\begin{equation}\label{egneq1}
	\begin{cases}
	 u\in W^{1,p}_0(\omega_2),\\
	 \displaystyle\int_{\omega_2}|A_{22}\nabla_{X_2} u\cdot\nabla_{X_2}u|^\frac{p-2}{2}A_{22}\nabla_{X_2} u\cdot\nabla_{X_2} v=\mu(\omega_2)\int_{\omega_2}|u|^{p-2}uv\;\;
	\forall v\in W^{1,p}_0(\omega_2).
	\end{cases}
	\end{equation}
	
	\noindent Remember that $\mu_1(\omega_2)$ and $W$ denote the first eigenvalue and the first normalized \mbox{($ ||W||_{L^p(\omega_2)} =1$)}  eigenfunction of the problem \eqref{egneq1} respectively. 
	
	As observed in the introduction, the first eigenvalue $\mu_1(\omega_2)$ has a variational characterization by the Rayleigh quotient:
	\begin{align*}\label{mixed cross section eigen value}
	\mu_1(\omega_2) &=\inf\bigg\{\int_{\omega_2}|A_{22}(X_2)\nabla_{X_2} u\cdot\nabla_{X_2} u|^{\frac{p}{2}}:u\in W^{1,p}_0(\omega_2),\int_{\omega_2}|u|^p=1\bigg\}\nonumber\\
	&=\inf_{\substack{u\in W^{1,p}_0(\omega_2)\\u\neq 0 }}\frac{\int_{\omega_2}|A_{22}(X_2)\nabla_{X_2} u\cdot\nabla_{X_2} u|^{\frac{p}{2}}}{\int_{\omega_2}|u|^p}.
	\end{align*}
	Moreover, $\mu_1(\omega_2)$ is simple and the eigenfunction $W$ is differentiable and has constant sign in the domain, that we should fix as the positive sign in the sequel. 
	
\begin{proof}[\textbf{Proof of Theorem \ref{conv for dirichlet case}}]
		
		By an abuse of notation we still denote with $W$ the extension of $W$ on $\mathbb{R}^m\times \omega_2$ defined by setting $W(X_1,X_2)=W(X_2)$.
		Then we have
		\begin{equation}\label{egn1}
		-\text{div}\big(|A\nabla W\cdot\nabla W|^\frac{p-2}{2}A\nabla W\big)=\mu_1(\omega_2)\;|W|^{p-2}W\;\;\;\text{  in }\Omega_\infty=\mathbb R^m\times\omega_2.
		\end{equation}
		Let $\phi$ be any function in $C_c^\infty(\Omega_\ell)$. We are now in position to use Picone's identity \ref{picone} because $W$ is $C^1$ and $W>0$, then we get
		
		$$
		|A\nabla(|\phi|)\cdot\nabla(|\phi|)|^\frac{p}{2}-\nabla\bigg(\frac{|\phi|^p}{W^{p-1}}\bigg)|A\nabla W\cdot\nabla W|^\frac{p-2}{2}A\nabla W\geq 0.
		$$
		Integrating over $\Omega_\ell$ and using Green's theorem we deduce
		\begin{equation}\label{green apply}
		\int_{\Omega_\ell}|A\nabla\phi\cdot\nabla\phi|^\frac{p}{2}-\int_{\Omega_\ell}\nabla\bigg(\frac{|\phi|^p}{W^{p-1}}\bigg)|A\nabla W\cdot\nabla W|^\frac{p-2}{2}A\nabla W\geq 0,
		\end{equation}

		$$\int_{\Omega_\ell}|A\nabla\phi\cdot\nabla\phi|^\frac{p}{2}\geq\int_{\Omega_\ell}-\text{div}\bigg(|A\nabla W\cdot\nabla W|^\frac{p-2}{2}A\nabla W\bigg)\frac{|\phi|^p}{W^{p-1}}.
		$$
		
		\noindent \text{Then using \eqref{egn1} we acquire}\\
		$$
		\displaystyle\int_{\Omega_\ell}|A\nabla\phi\cdot\nabla\phi|^\frac{p}{2} \geq\mu_1(\omega_2)\int_{\Omega_\ell}W^{p-1}\frac{|\phi|^p}{W^{p-1}}=\mu_1(\omega_2)\int_{\Omega_\ell}|\phi|^p.
		$$
		
		\noindent Since the last inequality holds true for any $\phi\in C_c^\infty(\Omega_\ell)$ we deduce by density in $W_0^{1,p}(\Omega_\ell)$ 
		$$
		\mu_1(\omega_2)\leq\inf_{\substack{\phi\in W^{1,p}_0(\Omega_\ell)\\ \phi\neq 0 }}\frac{\int_{\Omega_\ell}|A\nabla \phi\cdot\nabla \phi|^{\frac{p}{2}}}{\int_{\Omega_\ell}|\phi|^p}=\lambda^1_D(\omegal).
		$$
		
		The estimate from below for the eigenvalue $\lambda^1_D(\Omega_\ell)$ quoted in Theorem \ref{conv for dirichlet case} is then proved. \\
		To prove the estimate from above we use a suitable test function in the Rayleigh quotient characterizing $\lambda^1_D(\Omega_\ell)$. Let us choose $v_\ell$ be a smooth function in $W^{1,p}_0(\ell\omega_1)$ such that
		\begin{itemize}
			\item $v_\ell=1$ in $\frac{\ell}{2}\omega_1$;
			\item $0\leq v_\ell\leq 1$, $|\nabla v_\ell|\leq\frac{1}{\ell}$ everywhere.
		\end{itemize}
		Let $W$ be the first eigenfunction of the section problem as above. The function
		$$u_\ell(x)=v_\ell(X_1)W(X_2)\in W^{1,p}_0(\Omega_\ell)$$
		is a good test function in  \eqref{eigen value}.
		Thus, using Minkowski inequality and structure condition of the matrix $A$ we have\\
		\begin{align*}
		&\lambda^1_D(\omegal)\leq
		\frac{\displaystyle\int_{\Omega_\ell}|A\nabla(v_\ell W)\cdot\nabla(v_\ell W)|^\frac{p}{2}}{\displaystyle\int_{\Omega_\ell}|u_\ell|^p}
		\\
		&=
		\frac{\displaystyle\int_{\Omega_\ell}|(A_{11}\nabla_{X_1}v_\ell\cdot\nabla_{X_1}v_\ell)\;W^2+(2A_{12}\nabla_{X_2}W\cdot\nabla_{X_1}v_\ell)\;(v_\ell W)+(A_{22}\nabla_{X_2}W\cdot\nabla_{X_2}W)\;v_\ell^2|^\frac{p}{2}}{\displaystyle\int_{\ell\omega_1}|v_\ell|^p}
		\\
		&\leq
		\bigg\{\bigg(\int_{\Omega_\ell}|(A_{22}\nabla_{X_2}W\cdot\nabla_{X_2}W)\;v_\ell^2|^\frac{p}{2}\bigg)^{2/p}	+\bigg(\int_{\Omega_\ell}|(A_{11}\nabla_{X_1}v_\ell\cdot\nabla_{X_1}v_\ell)\;W^2|^\frac{p}{2}\bigg)^\frac{2}{p}
		\\
		&+
		\bigg(\int_{\Omega_\ell}|(2A_{12}\nabla_{X_2}W\cdot\nabla_{X_1}v_\ell)\;(v_\ell W)|^\frac{p}{2}\bigg)^\frac{2}{p}\bigg\}^\frac{p}{2}\bigg/\int_{\ell\omega_1}|v_\ell|^p.
		\end{align*}
		
		Where we also used the fact that $||W||_p=1$. Recalling that $W$ is an eigenfunction we deduce
		
		\begin{align*}	
		&\lambda^1_D(\omegal)\leq
		\bigg\{\mu_1(\omega_2)^{2/p}\bigg(\int_{\ell\omega_1}|v_\ell|^p\bigg)^\frac{2}{p}+||A_{11}||_{\infty}\bigg(\int_{\ell \omega_1}|\nabla_{X_1}v_\ell|^p\int_{\omega_2}|W|^p\bigg)^\frac{2}{p}
		\\
		&+
		2||A_{12}||_{\infty}\bigg(\int_{\ell \omega_1}|\nabla_{X_1}v_\ell)|^{\frac{p}{2}}\bigg)^\frac{2}{p}
		\bigg(\int_{\omega_2}|\nabla_{X_2}W)|^{\frac{p}{2}}|W|^\frac{p}{2}\bigg)^\frac{2}{p}\bigg\}^\frac{p}{2} \bigg/\int_{\ell\omega_1}|v_\ell|^p.
		\end{align*}	
		
		In the last estimate we also used the fact that $|v_\ell|\leq 1$. Now we observe that, since $W$ is an eigenfunction and thanks to the ellipticity condition on the matrix $A$, the following implication holds true
		
		$$
		\int_{\omega_2}|\nabla_{X_2}W|^p \leq \frac{\mu_1(\omega_2)}{\lambda}\implies\int_{\omega_2}|\nabla_{X_2} W|^{\frac{p}{2}}\leq C.
		$$
		
		Using the elementary inequality $(a+b)^q\leq a^q+q2^{q-1}(b^q+a^{q-1}b)$ for $a,b\geq 0$, $q\geq 1$ and denoting with $\mathcal{L}_m$ the Lebesgue measure in $\mathbb{R}^m$ we get	
		\begin{align*}
		\lambda^1_D(\omegal) 
		&\leq
		\Bigg\{\mu_1(\omega_2)^{2/p}+\frac{||A_{11}||_\infty\big(\int_{\ell\omega_1}|\nabla_{X_1}v_\ell|^p\big)^\frac{2}{p}+2C^{\frac{2}{p}}||A_{12}||_\infty \big(\int_{\ell\omega_1}|\nabla_{X_1}v_\ell|^\frac{p}{2}\big)^\frac{2}{p}}{||v_\ell||_{p,\ell\omega_1}^2}\Bigg\}^\frac{p}{2}
		\\
		&\leq
		\Bigg\{\mu_1(\omega_2)^{2/p}+\frac{||A_{11}||_\infty\bigg(\frac{\mathcal{L}_m(\ell\omega_1)}{\ell^p}\bigg)^{2/p}+2C^{\frac{2}{p}}||A_{12}||_\infty\bigg(\frac{\mathcal{L}_m(\ell\omega_1)}{\ell^{p/2}}\bigg)^{2/p}}{\mathcal{L}_m(\frac{\ell\omega_1}{2})^{2/p}}\Bigg\}^\frac{p}{2}
		\\
		&=
		\bigg\{\mu_1(\omega_2)^{2/p}+\frac{2^m||A_{11}||_\infty}{\ell^2}+\frac{2^{m+1}C^{\frac{2}{p}}||A_{12}||_\infty}{\ell}\bigg\}^\frac{p}{2}\\
		&
		\leq
		\bigg(\mu_1(\omega_2)^{2/p}+\frac{C_1}{\ell}\bigg)^\frac{p}{2}
		\leq\mu_1(\omega_2)+p2^{\frac{p-4}{2}}\bigg(\frac{C_1^{\frac{p}{2}}}{\ell^{\frac{p}{2}}}+\frac{C_1(\mu_1(\omega_2))^{\frac{p-2}{p}}}{\ell}\bigg).
	    \end{align*}
	   
	    Hence  the estimate from above is then proved.\smallskip\\

		Clearly, for any $\ell>0$, we have $\Lambda_\infty\leq\lambda_D^1(\omegal)$ and for the lower bound of $\Lambda_\infty$ we proceed exactly in the same way as it is done in \eqref{green apply} where $\omegal$ is replaced by $\Omega_\infty.$ Then letting $\ell\to\infty$ we conclude that $\Lambda_\infty=\mu_1(\omega_2).$ 
	\end{proof}
	\section{The Gap phenomenon for mixed boundary conditions}\label{mixed}
In this section we are concerned about the mixed boundary eigenvalue problem. Let us discuss some results that would be required to the main proof of the Theorem \ref{mixed case ell to infinity}. As we mentioned in the introduction, first we study the asymptotic behavior of $\lambda^1_M(\omegal)$ as $\ell\to 0$, which is a key ingredient to proof of the Theorem \ref{mixed case ell to infinity}.\smallskip\\
\noindent An appropriate space for mixed boundary eigenvalue problem is
 $$
 V(\omegal)=\{v\in W^{1,p}(\omegal):v=0\text{ on }\gamma_\ell\},
 $$
 where $\gamma_\ell=(-\ell,\ell)\times\partial\omega_2$ and the boundary value is defined in the sense of trace. Thanks to the classical Poincar\'e inequality, the space $V(\omegal)$ becomes a Banach space with respect to the norm \eqref{norm}.\smallskip\\

\noindent The weak formulation of the eigenvalue problem \eqref{mixed eigenvalue} is given by 
\begin{equation}\label{mixed weak form}
   \begin{cases}
	u\in V(\omegal),\\
    \displaystyle\int_{\omegal}|A\nabla u\cdot\nabla u|^\frac{p-2}{2}A\nabla u\cdot\nabla v=\lambda_M(\omegal)\displaystyle\int_{\omegal}|u|^{p-2}uv\;\;\text{ for any } v\in V(\omegal).
	\end{cases}
\end{equation}
The first eigenvalue $\lambda^1_M(\omegal)$ for \eqref{mixed weak form} is associated with a variational characterization
\begin{align}\label{mixed eigen value}
\lambda^1_M(\omegal)
&=
\inf\bigg\{\int_{\omegal}|A(X_2)\nabla u\cdot\nabla u|^{\frac{p}{2}}:u\in V(\omegal),\int_{\omegal}|u|^p=1\bigg\}\nonumber\\
&=
\inf_{\substack{u\in V(\omegal)\\u\neq 0 }}\frac{\int_{\omegal}|A(X_2)\nabla u\cdot\nabla u|^{\frac{p}{2}}}{\int_{\omegal}|u|^p}.
\end{align}
 Moreover, the first eigenvalue is simple and the corresponding eigenfunction has constant sign in the domain.
\begin{theorem}[\textbf{Dimension Reduction}]\label{mixed case ell to 0}
For $p\geq 2$, we have $\lim_{\ell\to 0}\lambda^1_M(\omegal)=\Lambda$ where 
\begin{equation*}\label{Lambda expression}
\Lambda=\inf\bigg\{\int_{\omega_2}\Bigg(A_{22}(X_2)\nabla u\cdot\nabla u-\frac{|A_{12}(X_2)\cdot\nabla u|^2}{a_{11}(X_2)}\Bigg)^{\frac{p}{2}}:u\in W^{1,p}_0(\omega_2),\int_{\omega_2}|u|^p=1\bigg\}.
\end{equation*}
\end{theorem}
\begin{proof}
The reason why we find $\Lambda$ as the limiting value will be clarified by the following observation. Let
$$ 
B =
\begin{pmatrix}
 b_{11} & B_{12}\\
 B_{12}^{t} & B_{22}
\end{pmatrix}
$$
 be a positive definite $n\times n$ matrix in $\rn$ and we write any vector $z=(z_1,Z_2)\in\rn$ with $Z_2\in\re^{n-1}.$ Then it is easy to see by using elementary calculus that for any fixed $Z_2$ we have
\begin{equation}\label{min gen form}
\min_{z_1\in\re}(Bz\cdot z)^{\frac{p}{2}}
=\Bigg(B_{22}Z_2\cdot Z_2-\frac{|B_{12}Z_2|^2}{b_{11}}\Bigg)^{\frac{p}{2}}
\end{equation}
and the minimum in \eqref{min gen form} is attained for $z_1=-\frac{B_{12}Z_2}{b_{11}}$. Applying \eqref{min gen form} with $B=A(X_2)$ we obtain, for any $\ell>0,$
\begin{align}\label{lower bound for l to 0}
    \int_{\omegal}|A(X_2)\nabla u_\ell\cdot\nabla u_\ell|^{\frac{p}{2}} &\geq\int_{\omegal}\Bigg(A_{22}(X_2)\nabla_{X_2} u_\ell\cdot\nabla_{X_2} u_\ell-\frac{|A_{12}(X_2)\cdot\nabla_{X_2} u_\ell|^2}{a_{11}(X_2)}\Bigg)^{\frac{p}{2}}\nonumber\\
    &\geq\Lambda\int_{\omegal}|u_\ell|^p.
\end{align}
 It is clear by \eqref{lower bound for l to 0} the lower bound 
 \begin{equation}\label{lower bound}
  \Lambda\leq\lim_{\ell\to 0}\inf\lambda^1_M(\omegal).   
 \end{equation}
 Let $T_\ell=\{x\in\omega_2:\;\text{dist}(x,\partial\omega_2)\leq\ell\}$ be a neighbourhood of $\partial\omega_2$ for $\ell>0$. Fix for any $\beta\in (0,1)$ and let $\rho_\ell$ be an approximation of the characteristic function of $\omega_2$, as $\ell\to 0$:
 \begin{equation}
 \label{properties rho ell}
    \rho_\ell\in C_c^{\infty}(\omega_2),\quad 0\leq \rho_{\ell}\leq 1,\quad 
    \rho_{\ell}=1\text{ in }\omega_2\setminus T_{\ell}, \quad 
     |\nabla\rho_{\ell}|\leq \frac{1}{\ell^{\beta}}\text{ in }T_{\ell}\,.
\end{equation}
Hence for $\ell\to 0$ one has $\rho_\ell\to 1$ pointwise.
 Let us define a function $u_\ell$ on $\omegal$ by 
 \begin{equation*}\label{special function}
     u_\ell(x)=W(X_2)-\frac{x_1\rho_\ell(X_2)A_{12}(X_2)\cdot\nabla W}{a_{11}(X_2)}.
 \end{equation*}
  We emphasize that the function $u_\ell$ defined above does not necessarily belong to the space $W^{1,p}$, since the first eigenfunction $W$ is atmost in $C^{1,\gamma}(\omega_2)$. To resolve this difficulty, we provide a smooth approximation argument, motivated by [\cite{br}, Ch.14]. Now define a family of functions  $\{F_\epsilon\}_{\epsilon>0}$ in $C_c^\infty(\omega_2)$ by using standard mollification which satisfies the following 
 \begin{align*}
     \lim_{\epsilon\to 0}F_\epsilon(X_2)=\frac{A_{12}(X_2)\cdot\nabla W}{a_{11}(X_2)}\text{ in }L^p(\omega_2).
 \end{align*}
 Then we define 
 \begin{align}\label{epsilon l function}
     u_\ell^\epsilon(x)=W(X_2)-x_1\;\rho_\ell(X_2)\;F_\epsilon(X_2).
 \end{align}
  By using simple elementary inequality for the vectors $a,\;b$ and $q\geq 1$ $$|b|^q\geq|a|^q+q\;\big\langle|a|^{q-2}a,\;b-a\big\rangle,$$ and
 using the fact that $x_1$ is an odd function in $(-\ell,\ell)$ we infer that
  \begin{equation}\label{lp norm estimate}
      \int_{\omegal}|u_\ell^\epsilon|^p=\int_{\omegal}|W(X_2)-x_1\rho_\ell(X_2)F_\epsilon(X_2)|^p\geq\int_{\omegal}|W(X_2)|^p=2\ell\int_{\omega_2}|W|^p.
  \end{equation}
 Now 
 \begin{align*}
   &
   \int_{\omegal}|A\nabla u_\ell^\epsilon\cdot\nabla u_\ell^\epsilon|^{\frac{p}{2}}\nonumber
   \\
   &=
   \int_{\omegal}|a_{11}(\partial_{x_1}u_\ell^\epsilon)^2+2(A_{12}\cdot\nabla_{X_2}u_\ell^\epsilon)\partial_{x_1}u_\ell^\epsilon+A_{22}\nabla_{X_2}u_\ell^\epsilon\cdot\nabla_{X_2}u_\ell^\epsilon|^{\frac{p}{2}}\nonumber
   \\
   &=
   \int_{\omegal}|a_{11}\rho_\ell^2 F_\epsilon^2-2\rho_\ell F_\epsilon(A_{12}\cdot\nabla W)-2x_1\rho_\ell F_\epsilon(F_\epsilon A_{12}\cdot\nabla\rho_\ell+\rho_\ell A_{12}\cdot\nabla F_\epsilon)\nonumber
   \\
   &\;\;\;\;\;\;\;\;\;\;+(A_{22}\nabla W-x_1(F_\epsilon A_{22}\cdot\nabla\rho_\ell+\rho_\ell A_{22}\cdot\nabla F_\epsilon))\cdot(\nabla W-x_1(F_\epsilon\nabla\rho_\ell+\rho_\ell\nabla F_\epsilon))|^{\frac{p}{2}}.
  \end{align*}
  Hence by using Minkowski inequality we have 
  \begin{equation}\label{gradient estimate} 
  \int_{\omegal}|A\nabla u_\ell^\epsilon\cdot\nabla u_\ell^\epsilon|^{\frac{p}{2}}\leq(I_1+I_2)^{\frac{p}{2}},
  \end{equation}
  where
  $$
  I_1:=\bigg(\int_{\omegal}|a_{11}F_\epsilon^2-2(A_{12}\cdot\nabla W)F_\epsilon+A_{22}\nabla W\cdot\nabla W|^{\frac{p}{2}}\bigg)^{\frac{2}{p}}
  $$
  and 
  $$
  I_2:=\bigg(\int_{\omegal}|a_{11}F^2_\epsilon(\rho_\ell^2-1)+2(1-\rho_\ell)(A_{12}\cdot\nabla W)F_\epsilon-2x_1 H^\epsilon_\ell(X_2)+x_1^2 G^\epsilon_\ell(X_2)|^{\frac{p}{2}}\bigg)^{\frac{2}{p}},
  $$
  where 
 \begin{equation*}
   H^\epsilon_\ell(X_2):=\rho_\ell F^2_\epsilon (A_{12}\cdot\nabla\rho_\ell)+\rho_\ell^2 F_\epsilon (A_{12}\cdot\nabla F_\epsilon)+F_\epsilon (A_{22}\nabla W\cdot\nabla\rho_\ell)+\rho_\ell (A_{22}\nabla W\cdot\nabla F_\epsilon),
  \end{equation*}
  \begin{equation*}
    G^\epsilon_\ell(X_2):=(F_\epsilon\;A_{22}\cdot\nabla\rho_\ell+\rho_\ell\;A_{22}\cdot\nabla F_\epsilon)\cdot(F_\epsilon\nabla\rho_\ell+\rho_\ell\nabla F_\epsilon).
  \end{equation*}
  Then by properties \eqref{properties rho ell} of the function $\rho_\ell$, and for fixed $\epsilon>0$ we have 
  \begin{equation}\label{functions estimate}
      |H^\epsilon_\ell(X_2)|\leq C_1+\frac{C_2}{\ell^\beta}\text{ and } |G^\epsilon_\ell(X_2)|\leq C_3+\frac{C_4}{\ell^{2\beta}},
  \end{equation}
  where $C_i\;(i=1,2,3,4)$ are positive constants independent of $\ell$, and we define 
   $$
   K^\epsilon_\ell=\int_{\omega_2}|a_{11}F^2_\epsilon(\rho_\ell^2-1)+2(1-\rho_\ell)(A_{12}\cdot\nabla W)F_\epsilon|^{\frac{p}{2}}.
   $$
  Since $\rho_\ell\to 1$ pointwise as $\ell\to 0$ and then by dominated convergence theorem we conclude that $K^\epsilon_\ell\to 0$ as $\ell\to 0.$ Now we estimates the above integrals $I_1,\;I_2$ in the following:
  
  \noindent\textit{Estimate for $I_1$:}
  \begin{equation}\label{calculation of I_1}
     I_1\leq(2\ell)^{\frac{2}{p}}\;\bigg(\int_{\omega_2}|a_{11}F_\epsilon^2-2(A_{12}\cdot\nabla W)F_\epsilon+A_{22}\nabla W\cdot\nabla W|^{\frac{p}{2}}\bigg)^{\frac{2}{p}}. 
  \end{equation}
  \noindent\textit{Estimate for $I_2$:} Again applying Minkowski inequality and by \eqref{functions estimate} we obtain
  \begin{align}\label{calculation of I2}
     I_2
      &\leq\bigg(\int_{\omegal}|a_{11}F^2_\epsilon(\rho_\ell^2-1)+2(1-\rho_\ell)(A_{12}\cdot\nabla W)F_\epsilon|^{\frac{p}{2}}\bigg)^{\frac{2}{p}}+\bigg(\int_{\omegal}|x_1^2 G^\epsilon_\ell(X_2)-2x_1 H^\epsilon_\ell(X_2)|^{\frac{p}{2}}\bigg)^{\frac{2}{p}}\nonumber
      \\
      &\leq
      (2\ell)^{\frac{2}{p}} (K_\ell^\epsilon)^{\frac{2}{p}}+\bigg(\int_{\omegal}|x_1^2 G^\epsilon_\ell(X_2)|^{\frac{p}{2}}\bigg)^{\frac{2}{p}}+\bigg(\int_{\omegal}|2x_1 H^\epsilon_\ell(X_2)|^{\frac{p}{2}}\bigg)^{\frac{2}{p}}\nonumber\\
      &\leq
       (2\ell)^{\frac{2}{p}}  (K_\ell^\epsilon)^{\frac{2}{p}}+\ell^2\bigg(\int_{\omegal}| G^\epsilon_\ell(X_2)|^{\frac{p}{2}}\bigg)^{\frac{2}{p}}+2\ell\bigg(\int_{\omegal}| H^\epsilon_\ell(X_2)|^{\frac{p}{2}}\bigg)^{\frac{2}{p}}\nonumber
       \\
       &=
        (2\ell)^{\frac{2}{p}}\bigg[  (K_\ell^\epsilon)^{\frac{2}{p}}+\ell^2\bigg(\int_{\omega_2}| G^\epsilon_\ell(X_2)|^{\frac{p}{2}}\bigg)^{\frac{2}{p}}+2\ell\bigg(\int_{\omega_2}| H^\epsilon_\ell(X_2)|^{\frac{p}{2}}\bigg)^{\frac{2}{p}}\bigg]\nonumber
        \\
        &\leq
        (2\ell)^{\frac{2}{p}}\bigg[ (K_\ell^\epsilon)^{\frac{2}{p}}+C_1\ell^2+C_2\ell^{2-2\beta}+2C_3\ell+2C_4\ell^{1-\beta}\bigg]=:(2\ell)^{\frac{2}{p}}\bigg( (K_\ell^\epsilon)^{\frac{2}{p}}+C(\ell)\bigg).
      \end{align}
  Now plugging the estimates \eqref{calculation of I_1}, \eqref{calculation of I2} into \eqref{gradient estimate}
  we obtain 
  \begin{align}\label{final gradient estimate}
    \int_{\omegal}|A\nabla u_\ell^\epsilon\cdot\nabla u_\ell^\epsilon|^{\frac{p}{2}}
    \leq 2\ell\bigg[\bigg(\int_{\omega_2}|a_{11}F_\epsilon^2-2(A_{12}\cdot\nabla W)F_\epsilon+A_{22}\nabla W\cdot\nabla W|^{\frac{p}{2}}\bigg)^{\frac{2}{p}}+ (K_\ell^\epsilon)^{\frac{2}{p}}+C(\ell)\bigg]^{\frac{p}{2}}.
  \end{align}
 Therefore combining \eqref{lp norm estimate} and \eqref{final gradient estimate} we have 
 \begin{align*}
     &\lim_{\ell\to 0}\sup\lambda^1_M(\omegal)\nonumber\\ 
     &\leq
     \lim_{\ell\to 0}\sup\frac{\displaystyle\int_{\omegal}|A\nabla u_\ell^\epsilon\cdot\nabla u_\ell^\epsilon|^{\frac{p}{2}}}{\displaystyle\int_{\omegal}|u_\ell^\epsilon|^p}\nonumber
     \\
     &\leq
     \lim_{\ell\to 0}\sup\bigg[\bigg(\displaystyle\int_{\omega_2}|a_{11}F_\epsilon^2-2(A_{12}\cdot\nabla W)F_\epsilon+A_{22}\nabla W\cdot\nabla W|^{\frac{p}{2}}\bigg)^{\frac{2}{p}}+ (K_\ell^\epsilon)^{\frac{2}{p}}+C(\ell)\bigg]^{\frac{p}{2}}\nonumber
     \\
     &=
     \displaystyle\int_{\omega_2}|a_{11}F_\epsilon^2-2(A_{12}\cdot\nabla W)F_\epsilon+A_{22}\nabla W\cdot\nabla W|^{\frac{p}{2}}.
 \end{align*}
 Letting $\epsilon\to 0$ and using the fact $F^2_{\epsilon}\to\frac{|A_{12}(X_2)\cdot\nabla W|^2}{a_{11}(X_2)^2}$ in $L^{\frac{p}{2}}(\omega_2)$, we infer that
 $$
 \lim_{\ell\to 0}\sup\lambda^1_M(\omegal)\leq\displaystyle\int_{\omega_2}\Bigg(A_{22}(X_2)\nabla W\cdot\nabla W-\frac{|A_{12}(X_2)\cdot\nabla W|^2}{a_{11}(X_2)}\Bigg)^{\frac{p}{2}}=\Lambda,
 $$
 which together with \eqref{lower bound} gives the desired result.
\end{proof}
\smallskip
\begin{proof}[\textbf{Proof of Theorem \ref{mixed case ell to infinity}}:]
\textit{Case 1:}
Suppose the condition holds first i.e. $A_{12}\cdot\nabla W\neq 0$ a.e. on $\omega_2$. Then we obtain
$$
\Lambda\leq\displaystyle\int_{\omega_2}\Big|A_{22}(X_2)\nabla W\cdot\nabla W-\frac{|A_{12}(X_2)\cdot\nabla W|^2}{a_{11}(X_2)}\Big|^{\frac{p}{2}}<\int_{\omega_2}|A_{22}(X_2)\nabla W\cdot\nabla W|^{\frac{p}{2}}=\mu_1(\omega_2).
$$
By the proof of the above theorem there exists $\ell_0>0$ and $\epsilon_0>0$ such that the function $u_{\ell_0}^{\epsilon_0}$ defined by \eqref{epsilon l function} satisfies
\begin{equation}\label{ell0 domain estimate}
    \int_{\Omega_{\ell_0}}|A\nabla u_{\ell_0}^{\epsilon_0}\cdot\nabla u_{\ell_0}^{\epsilon_0}|^{\frac{p}{2}}<\mu_1(\omega_2)\int_{\Omega_{\ell_0}}|u_{\ell_0}^{\epsilon_0}|^p.
\end{equation}
Let $\alpha>1$ be a constant whose value will be choose later. For $\ell>\ell_0+\alpha$ we define a function $\phi_{\ell}$ as follows,
\begin{equation*}
\phi_{\ell}(x_1,X_2)= \begin{cases}
                      u_{\ell_0}^{\epsilon_0}(x_1-\ell+\ell_0,X_2) & \text{in }(\ell-\ell_0,\ell)\times\omega_2,\\
                      \frac{\xi(x_1)\;W(X_2)}{\alpha} & \text{ in }(\ell-\ell_0-\alpha,\ell-\ell_0)\times\omega_2,\\
                      0 & \text{ in }\Omega_{\ell-\ell_0-\alpha},\\
                      \frac{\xi(x_1)\;W(X_2)}{\alpha} & \text{ in }(\ell_0-\ell,-(\ell-\ell_0-\alpha))\times\omega_2,\\
                      u_{\ell_0}^{\epsilon_0}(x_1+\ell-\ell_0,X_2) & \text{ in }(-\ell,\ell_0-\ell)\times\omega_2,
                       \end{cases}
\end{equation*}
where 
\begin{equation*}
    \xi(x_1)=\begin{cases}x_1-\ell+\ell_0+\alpha & \text{ if }x_1\in (\ell-\ell_0-\alpha,\ell-\ell_0),\\
    -x_1-\ell+\ell_0+\alpha & \text{ if }x_1\in(\ell-\ell_0,-\ell+\ell_0+\alpha).
    \end{cases}
\end{equation*}
By simple change of variable we get 

$$
\int_{\omegal\setminus\Omega_{\ell-\ell_0}}|\phi_\ell|^p=\int_{\Omega_{\ell_0}}|u_{\ell_0}^{\epsilon_0}|^p
$$ 
and thus we have 
\begin{equation}\label{lp norm of phi}
\int_{\omegal}|\phi_\ell|^p=\int_{\Omega_{\ell_0}}|u_{\ell_0}^{\epsilon_0}|^p+\frac{2\alpha}{p+1}.
\end{equation}
Similarly we have
\begin{equation}\label{final estimate}
  \int_{\omegal}|A\nabla\phi_\ell\cdot\nabla\phi_\ell|^{\frac{p}{2}}
  = \int_{\Omega_{\ell_0}}|A\nabla u_{\ell_0}^{\epsilon_0}\cdot\nabla u_{\ell_0}^{\epsilon_0}|^{\frac{p}{2}}+\int_{\Omega_{\ell-\ell_0}}|A\nabla\phi_\ell\cdot\nabla\phi_\ell|^{\frac{p}{2}}.
 \end{equation}
 Let $S=\Omega_{\ell-\ell_0}\setminus\Omega_{\ell-\ell_0-\alpha}$ and $S^+=(\ell-\ell_0-\alpha, \ell-\ell_0)\times\omega_2$. Using Minkowski inequality and the fact that $\phi_\ell$ is an even function of $x_1$ on $S$. We estimate the above last integral as follows
 \begin{align}\label{annular estimate}
   &
   \int_{\Omega_{\ell-\ell_0}}|A\nabla\phi_\ell\cdot\nabla\phi_\ell|^{\frac{p}{2}}\nonumber\\
   &=
   \int_{S}|a_{11}(X_2)(\partial_{x_1}\phi_{\ell})^2+A_{22}\nabla_{X_2}\phi_{\ell}\cdot\nabla_{X_2}\phi_{\ell}+2(A_{12}\cdot\nabla_{X_2}\phi_{\ell})(\partial_{x_1}\phi_{\ell})|^{\frac{p}{2}}\nonumber\\
  &=
  \frac{1}{\alpha^p}\int_{S}|(A_{22}\nabla W\cdot\nabla W)\;\xi^2+a_{11}(X_2)W^2+2(A_{12}\cdot\nabla W)\xi\;\xi^\prime|^{\frac{p}{2}}\nonumber\\
  &\leq
  \frac{1}{\alpha^p}\bigg\{\bigg(\int_{S}|A_{22}\nabla W\cdot\nabla W|^{\frac{p}{2}}|\xi|^p\bigg)^{\frac{2}{p}}+\bigg(\int_{S}|a_{11}(X_2)W^2+2(A_{12}\cdot\nabla W)\xi\;\xi^\prime|^{\frac{p}{2}}\bigg)^{\frac{2}{p}}\bigg\}^{\frac{p}{2}}\nonumber\\
  &=
  \frac{1}{\alpha^p}(I_1+I_2)^{\frac{p}{2}},
 \end{align}
where
\begin{align}
    I_1
    =\bigg(\int_{S}|A_{22}\nabla W\cdot\nabla W|^{\frac{p}{2}}|\xi|^p\bigg)^{\frac{2}{p}}
    &=
    \bigg(2\int_{S^+}|A_{22}\nabla W\cdot\nabla W|^{\frac{p}{2}}\;\xi^p\bigg)^{\frac{2}{p}}\nonumber\\
    &=
    \bigg(2\mu_1(\omega_2)\int_{\ell-\ell_0-\alpha}^{\ell-\ell_0}(x_1-\ell+\ell_0+\alpha)^p\;dx_1\bigg)^{\frac{2}{p}}\nonumber\\
    &=
    \bigg(\frac{2\alpha^{p+1}\mu_1(\omega_2)}{p+1}\bigg)^{\frac{2}{p}}\nonumber,
\end{align}
and
$I_2=\bigg(\int_{S}|a_{11}(X_2)W^2+2(A_{12}\cdot\nabla W)\xi\;\xi^\prime|^{\frac{p}{2}}\bigg)^{\frac{2}{p}}$,
since $||a_{11}(X_2)W^2+2(A_{12}\cdot\nabla W)\xi\;\xi^\prime||_{\infty}\leq M$ for some $M>0$, then we have 
$I_2\leq C\;\alpha^{\frac{2}{p}}.$\\
 Now plugging the above estimates into \eqref{annular estimate} and then using the elementary inequality which stated in section \ref{dirichlet} we obtain
  \begin{equation}\label{annular domain estimate}
       \int_{\Omega_{\ell-\ell_0}}|A\nabla\phi_\ell\cdot\nabla\phi_\ell|^{\frac{p}{2}}
       \leq
       \frac{1}{\alpha^p}\bigg(\bigg(\frac{2\alpha^{p+1}\mu_1(\omega_2)}{p+1}\bigg)^{\frac{2}{p}}+C\alpha^{\frac{2}{p}}\bigg)^{\frac{p}{2}}
       \leq
       \frac{2\alpha\mu_1(\omega_2)}{p+1}+\frac{C_1}{\alpha^{p-1}}+\frac{C_2}{\alpha}.
  \end{equation}
  Combining \eqref{lp norm of phi}, \eqref{final estimate} and \eqref{annular domain estimate} we get
  
  $$
  \lambda^1_M(\omegal)
  \leq
  \frac{\int_{\omegal}|A\nabla\phi_\ell\cdot\nabla\phi_\ell|^{\frac{p}{2}}}{\int_{\omegal}|\phi_\ell|^p}
  \leq
  \frac{\int_{\Omega_{\ell_0}}|A\nabla u_{\ell_0}^{\epsilon_0}\cdot\nabla u_{\ell_0}^{\epsilon_0}|^{\frac{p}{2}}+\frac{2\alpha\mu_1(\omega_2)}{p+1}+\frac{C_3}{\alpha}}{\int_{\Omega_{\ell_0}}|u_{\ell_0}^{\epsilon_0}|^p+\frac{2\alpha}{p+1}}.
  $$
  
 \noindent Therefore,
  \begin{equation*}
      \lambda^1_M(\omegal)-\mu_1(\omega_2)
      \leq
      \frac{\int_{\Omega_{\ell_0}}|A\nabla u_{\ell_0}^{\epsilon_0}\cdot\nabla u_{\ell_0}^{\epsilon_0}|^{\frac{p}{2}}-\mu_1(\omega_2)\int_{\Omega_{\ell_0}}|u_{\ell_0}^{\epsilon_0}|^p+\frac{C_3}{\alpha}}{\int_{\Omega_{\ell_0}}|u_{\ell_0}^{\epsilon_0}|^p+\frac{2\alpha}{p+1}}.
  \end{equation*}
  By \eqref{ell0 domain estimate} it is clear that for a fixed large enough $\alpha$ such that the RHS of the above is negative and get the desired result.\smallskip\\
  
\noindent\textit{Case 2:} Suppose the condition of the Theorem doesn't hold i.e. $A_{12}\cdot\nabla W=0$ in $\omega_2$. Then $\Lambda$ becomes $\mu_1(\omega_2)$ and by \eqref{lower bound for l to 0} we conclude that $\mu_1(\omega_2)\leq\lambda^1_M(\omegal)\;\forall\ell>0$. Now by choosing $u(x)=W(X_2)$ as a test function in \eqref{mixed eigen value} then we get $\lambda^1_M(\omegal)\leq\mu_1(\omega_2).$ This completes the proof of the theorem.
\end{proof}

\textbf{Acknowledgement:} The second author would like to thank Prof. Itai Shafrir for several discussion on the subject, during his stay in Technion, Israel. We also thank Prof. Itai Shafrir for suggesting us the use of Picone's identity in the proof of Theorem \ref{conv for dirichlet case}. The work of the second author is supported by INSPIRE grant IFA14-MA43 and Matrix grant MTR/2019/000585.

\end{document}